\documentclass[11pt,reqno,a4paper]{amsart}
\usepackage{amsfonts}
\usepackage{ifthen}
\usepackage{amsthm,cite}
\usepackage{amsmath,mathrsfs}
\usepackage{graphicx}
\usepackage{amscd,amssymb,amsthm}
\usepackage{color}
\usepackage{hyperref}
\usepackage{amsthm}
\usepackage{amsmath}

\allowdisplaybreaks

\usepackage{array,multirow,makecell}

\setcellgapes{1pt}
\makegapedcells
\newcolumntype{R}[1]{>{\raggedleft\arraybackslash }b{#1}}
\newcolumntype{L}[1]{>{\raggedright\arraybackslash }b{#1}}
\newcolumntype{C}[1]{>{\centering\arraybackslash }b{#1}}
\newcounter{minutes}
\divide\time by 60
\newcounter{hours}
\multiply\time by 60 \addtocounter{minutes}{-\time}

\newtheorem{theo}{Theorem}[section]

\newtheorem{theorem}{Theorem}[section]

\newtheorem{lemma}[theorem]{Lemma}

\newtheorem{corollary}[theorem]{Corollary}
\newtheorem{remark}[theorem]{Remark}

\title{ Properties and applications of the Bicomplex Miller-Ross function}
\author[Snehasis Bera]{Snehasis Bera}
\address{Snehasis Bera\newline Department of Mathematics,\newline National Institute of Technology Jamshedpur, Jamshedpur-831014, Jharkhand, India.}
\email{berasnehasis1996@gmail.com}
\author[Sourav Das]{Sourav Das$^\ast$}
\thanks{$^\ast$Corresponding author}
\address{Sourav Das\newline Department of Mathematics,\newline National Institute of Technology Jamshedpur, Jamshedpur-831014, Jharkhand, India.}
\email{souravdasmath@gmail.com, souravdas.math@nitjsr.ac.in}
\author[Abhijit Banerjee]{Abhijit Banerjee}
\address{Abhijit Banerjee\newline Department of Mathematics,\newline Garhbeta College, Paschim Medinipur-721127, West Bengal, India.}
\email{abhijit.banerjee.81@gmail.com}
\keywords{Bicomplex numbers, Bicomplex functions, Miller Ross function, Kinetic equation}
\subjclass{30G35; 33C20}
\date{}
\begin{document}

\begin{abstract}
  In this work, Miller Ross function with bicomplex arguments has been introduced. Various properties of this function including recurrence relations, integral representations and differential relations are established. Furthermore, the bicomplex holomorphicity and Taylor series representation of this function are discussed, along with the derivation of a differential equation.
  Finally, as applications some relations of fractional order derivatives and solutions for the bicomplex extension of the generalized fractional kinetic equation involving the bicomplex Miller Ross function are derived.
\end{abstract}

\maketitle

\section{Introduction and motivation}
 In 1892, C. Segre \cite{le}  introduced the concept of bicomplex numbers, considering them as a generalization of complex numbers. Similar to how complex numbers are perceived as two-dimensional vectors, bicomplex numbers can be thought of as four-dimensional vectors. Moreover, they can be conceptualized as ordered pairs of two complex numbers. The set of bicomplex numbers, denoted as $\mathbb{BC}$, forms a real vector space equipped with operations of addition and scalar multiplication. Numerous properties of bicomplex numbers have been uncovered, encompassing algebraic and
geometric aspects, as well as practical applications  (see \cite{bc_number,multicomplex,hyperbolic}).

Miller Ross function is one of the important special functions, playing a crucial role in fractional calculus \cite{kernel}, statistics and
probability theory\cite{subclass} and related branches of science
and engineering.  In \cite{gamma}, the authors extend gamma and beta functions from a complex variable to a bicomplex variable, establishing various properties, including the Gauss multiplication theorem, binomial theorem, and Legendre duplication formula. In \cite{polygamma}, the authors introduced a bicomplex generalization of the polygamma function, exploring properties such as recurrence relations, multiplication formula, reflection formula and integral representations.
Furthermore, in \cite{bc_hypergeometric}, the authors expanded the domain of the hypergeometric function from complex space to bicomplex space, establishing integral representations and recurrence relations. The bicomplex extension of several special functions and their applications can be found in (see \cite{zeta,bc_zeta,neural}). Motivated by the above results, introduce Miller-Ross function with bicomplex arguments.
The set of bicomplex numbers $\mathbb{BC}$ is defined by \cite{bc_number}:
\begin{align*}\mathbb{BC}&=\{Z=a_1+ia_2+ja_3+ka_4|a_1,a_2,a_3,a_4\in \mathbb{R}\}\\
&=\{Z=w_1+jw_2|w_1,w_2\in \mathbb{C}(i)\},
\end{align*}
where $i,j$ and $k$ satisfies the conditions: $i^2=j^2=-1,k^2=1$ and $ij=ji=k$.
For every bicomplex number $Z=w_1+jw_2\in \mathbb{BC} $ has unique representation as follows :  $$ Z=(w_1-iw_2)e_1+(w_1+iw_2)e_2=z_1e_1+z_2 e_2,$$ where $e_1=\frac{1+k}{2}$  and \; $e_2=\frac{1-k}{2}$, which satisfies the identities: $e_1+e_2=1, e_1-e_2=ij, e_1^{2}=e_1, e_2^{2}=e_2$ and this representation is called idempotent representation of the bicomplex number $Z\in\mathbb{BC}$. Idempotent representation is very helpful as addition, multiplication and division can be done term by term. Specifically, let $Z_1=z_{11} e_1+z_{12} e_2$ and $Z_2=z_{21} e_1+z_{22} e_2$ be two bicomplex numbers, then
\begin{align*}
    &Z_1+Z_2 =(z_{11}+z_{21})e_1+(z_{12}+z_{22})e_2,\\
    &Z_1\cdot Z_2=(z_{11}z_{21})e_1+(z_{12}z_{22}) e_2,\\
   & \frac{Z_1}{Z_2}=\frac{z_{11}}{z_{21}}e_1+\frac{z_{21}}{z_{22}}e_2; \quad Z_2\notin \mathbb{O}_2.
\end{align*}
The space $\mathbb{BC}$ forms a commutative unitary ring with zero divisors
 and the set of all zero divisors is defined by \cite{bc_number} : $$\mathbb{O}_2=\{Z=w_1+jw_2|{w_1}^2+{w_2}^2=0\}.$$ Moreover, the elements of $\mathbb{O}_2$ are also called singular elements.

For any bicomplex number $Z=z_1e_1+z_2 e_2$, the hyperbolic norm is defined as\cite{bc_number}: $$|Z|_h=|z_1|e_1+|z_2|e_2,$$ and satisfies $|ZW|_h=|Z|_h|W|_h$, for any $Z,W\in\mathbb{BC}$.

The definition of the derivative at $Z_0$ of a bicomplex function $f:\Omega\subset \mathbb{BC} \rightarrow \mathbb{BC}$ is similar to the derivative of a complex valued function  i,e,  if the limit $$\lim\limits_{\substack{{\Delta Z\to 0}\\{\Delta Z\notin \mathbb{O}_2}}}\frac{f(Z_0+\Delta Z)-f(Z_0)}{\Delta Z}$$ exist finitely. Furthermore, if $f$ is differentiable at every points in $\Omega$ then $f$ is said to be bicomplex holomorphic in $\Omega$. With the help of the theorem of function of complex variable this is equivalent to say that $f=f_1(w_1,w_2)+jf_2(w_1,w_2)$ is bicomplex holomorphic if and only if, $f_1$ and $f_2$ are holomorphic in $\Omega$ and satisfies bicomplex Cauchy-Riemann equations :
\begin{align}\label{eq:i5}
       \frac{\partial f_1}{\partial w_1}= \frac{\partial f_2}{\partial w_2} \quad
      and \quad  \frac{\partial f_1}{\partial w_2}=- \frac{\partial f_2}{\partial w_1}.
      \end{align}

Let $D$ be a piece wise continuously differentiable curve in $\mathbb{BC}$ defined by $D:\xi(t),\; a\leq t\leq b$, then integration of bicomplex function $f$ is given by \cite{multicomplex}
\begin{align*}
    \int_D f(\xi) d\xi=\int_a^b f(\xi(t)) \xi'(t) dt.
\end{align*}
Moreover, if $D_1:\xi_1(t),\;a\leq t\leq b$ and $D_2:\xi_2(t),\;a\leq t\leq b$ are two curves in $\mathbb{C}(i)$ such that $D:\xi(t)=\xi_1(t)e_1+\xi_2(t)e_2,\;a\leq t\leq b$ i.e. $D=(D_1,D_2)$, then
\begin{align*}
    \int_D f(\xi) d\xi=\left[\int_{D_1}f(\xi_1) d\xi_1\right] e_1+\left[\int_{D_2}f(\xi_2) d\xi_2\right] e_2.
\end{align*}
For every $Y=x_1+jx_2=y_1e_1+y_2e_2\in{\mathbb{BC}}$, Euler product form of the bicomplex gamma function is defined as follows \cite{gamma}:
    \begin{align*}
        \Gamma(Y)=\frac{1}{Ye^{\gamma Y}}\prod_{n=1}^{\infty}\left(\left(\frac{n}{n+Y}\right)\exp\left(\frac{Y}{n}\right)\right),
    \end{align*}
   with $x_1\not=-\frac{(p+q)}{2}$, $x_2\not=i\frac{(q-p)}{2}$ where  $p,q\in \mathbb{N}\cup \{0\}$ and $\gamma$ is the Euler constant \cite{sp_function}. Furthermore, bicomplex gamma function can be expressed as
    \begin{align}\label{eq:g}
        \Gamma(Y)=\Gamma(y_1) e_1+ \Gamma(y_2) e_2.
    \end{align}
 Miller Ross function is defined as \cite{miller} :
    \begin{align}\label{eq:5}
         E_{\nu,c}(z)=z^\nu
  \sum_{n=0}^{\infty}\frac{(cz)^n}{\Gamma{(\nu+n+1)}}, \;\;(\nu,c,z\in{\mathbb{C}}).
    \end{align}
Let us recall the following lemmas, which will be helpful to prove the main results.
\begin{lemma}\rm{\cite{bicomplex ghf}}
    Let $p\leq q+1$ and $W(Z)={}_{p} F_{q}\left[\begin{matrix}&\alpha_1, &\alpha_2,.... &\alpha_{p};\\&\beta_1,&\beta_2,.....&\beta_{q};&\end{matrix}Z\right]$ be a bicomplex generalized hypergeometric function then it satisfies the differential equation
    \begin{align}\label{eq:in}
        \frac{d}{dZ} \prod\limits_{j=1}^{q}\left(Z\frac{d}{dZ}+\beta_j-1\right)W(Z)-\prod\limits_{i=1}^{p}\left(Z\frac{d}{dZ}+\alpha_i\right)W(Z)=0.
    \end{align}
\end{lemma}
\begin{lemma}\rm{\cite{sp_function}}\label{th:zd}
    If $a_m$ and $b_j$ are neither zero nor negative integer and $\Re(z)<0$, the Barnes type contour integral of the generalized hypergeometric function is given by
    \begin{align*}
        \int_B\frac{(-z)^s\Gamma(-s)\prod\limits_{m=1}^q\Gamma(a_m+s)}{\prod\limits_{j=1}^q\Gamma(b_j+s)}ds=\frac{2\pi i \prod\limits_{m=1}^q\Gamma(a_m)}{\prod\limits_{j=1}^q\Gamma(b_m)}{}_{q} F_{q}\left[\begin{matrix}&a_1, &a_2,.... &a_{p};\\&b_1,&b_2,.....&b_{q};&\end{matrix}z\right],
    \end{align*}
    where $B$ is a Barnes path.
\end{lemma}
\section{Bicomplex Miller Ross function}
In this section, we introduce bicomplex Miller Ross function and discuss its convergence in bicomplex space. Moreover, we establish various recurrence relations of this function. The bicomplex Miller-Ross function is defined as
\begin{align}\label{eq:gh}
    E_{\mathcal{V},C}(Z)=Z^\mathcal{V}
  \sum_{r=0}^{\infty}\frac{(CZ)^r}{\Gamma{(\mathcal{V}+r+1)}},
\end{align}
where $\mathcal{V}=\nu_1e_1+\nu_2e_2=\alpha+j\beta,C=c_1e_1+c_2e_2,Z=z_1e_1+z_2e_2\in{\mathbb{BC}}$.
The following theorem justifies our definition of bicomplex Miller Ross function.
\begin{theo}
    Suppose that  $Z=z_1e_1+z_2e_2, \mathcal{V}=\nu_1e_1+\nu_2e_2=\alpha +j\beta, C=c_1e_1+c_2e_2 \in{\mathbb{BC}}$, where $e_1=\frac{1+k}{2}$, $e_2=\frac{1-k}{2}$, then idempotent representation of the bicomplex Miller Ross function can be expressed as
 \begin{align}\label{eq:fa}
     E_{\mathcal{V},C}(Z)=E_{\nu_1,c_1}(z_1)e_1+ E_{\nu_2,c_2}(z_2)e_2,
 \end{align}
  and the series $Z^{-\mathcal{V}} E_{\mathcal{V},C}(Z)$ absolutely hyperbolic  convergent for all $Z\in\mathbb{BC}$.
 \end{theo}
 \begin{proof}
 Using \eqref{eq:g} and \eqref{eq:gh}, we have
\begin{align}
    E_{\mathcal{V},C}(Z)&=(z_1e_1+z_2e_2)^{(\nu_1e_1+\nu_2e_2)} \sum_{r=0}^{\infty}\frac{(c_1z_1e_1+c_2z_2e_2)^r}{\Gamma{(\nu_1e_1+\nu_2e_2+r+1)}}\nonumber\\&=(z_1^{\nu_1}e_1+z_2^{\nu_2}e_2) \sum_{r=0}^{\infty}\frac{(c_1z_1)^re_1+(c_2z_2)^re_2}{\Gamma{(\nu_1+r+1)}e_1+\Gamma{(\nu_2+r+1)}e_2}\nonumber\\&=(z_1^{\nu_1}e_1+z_2^{\nu_2}e_2) \left[\sum_{r=0}^{\infty}\frac{(c_1z_1)^r}{\Gamma{(\nu_1+r+1)}}e_1+\sum_{r=0}^{\infty}\frac{(c_2z_2)^r}{\Gamma{(\nu_2+r+1)}}e_2\right]\nonumber\\
    &={z_1}^{\nu_1}\sum_{r=0}^{\infty}\frac{(c_1z_1)^r}{\Gamma{(\nu_1+r+1)}}e_1+{z_2}^{\nu_2}\sum_{r=0}^{\infty}\frac{(c_2z_2)^r}{\Gamma{(\nu_2+r+1)}}e_2\nonumber\\
    &=E_{\nu_1,c_1}(z_1)e_1+ E_{\nu_2,c_2}(z_2)e_2.\nonumber
\end{align}
Now, the idempotent components of $\mathcal{V}, C, Z\in\mathbb{BC}$ all are complex numbers then complex Miller-Ross functions $ E_{\nu_1,c_1}(z_1)$ and  $E_ {\nu_2,c_2}(z_2)$ are well defined and any bicomplex Miller-Ross functions can be represented as $E _ {\mathcal{V},C}(Z) =E_{\nu_1,c_1}(z_1)e_1 +E_ {\nu_2,c_2}(z_2)e_2$, that implies bicomplex Miller Ross function $E _ {\mathcal{V},C}(Z)$ is well define and this representation is idempotent representation of bicomplex Miller Ross function $E _ {\mathcal{V},C}(Z)$. Now, we will check convergence of the series $Z^{-\mathcal{V}} E_{\mathcal{V},C}(Z)$ defined in \eqref{eq:gh}, using hyperbolic ratio test \cite{bc_root}. Let $R=R_1e_1+R_2e_2$ be the hyperbolic radius of convergent of the bicomplex power series
\begin{align*}
    Z^{-\mathcal{V}} E_{\mathcal{V},C}(Z)=\sum_{r=0}^{\infty}\frac{(CZ)^n}{\Gamma{(\mathcal{V}+r+1)}} =\sum_{r=0}^{\infty}a_rZ^r=\sum_{r=0}^{\infty}a_{1r}z_1^re_1+\sum_{r=0}^{\infty}a_{2r}z_2^re_2,
\end{align*}
where $a_{1r}=\frac{c_1^r}{\Gamma(\nu_1+r+1)}$ and $a_{2r}=\frac{c_2^r}{\Gamma(\nu_2+r+1)}$. Then,
\begin{align}\label{eq:mi}
     R=R_1e_1+R_2e_2&=\lim_{r\to \infty} \sup\frac{|a_r|_h}{|a_{r+1}|_h}\nonumber\\
    &=\lim_{n\to \infty} \sup\frac{|a_{1r}|}{|a_{1r+1}|}e_1+\lim_{r\to \infty} \sup\frac{|a_{2r}|}{|a_{2r+1}|}e_2\nonumber\\
    &=\lim_{r\to \infty} \sup\left|\frac{(\nu_1+r+1)}{c_1}\right | e_1+ \lim_{r\to \infty} \sup\left|\frac{(\nu_2+r+1)}{c_2}\right | e_2.
\end{align}
From \eqref{eq:mi}, we have $R_1=\infty$ and $R_2=\infty$. Then by hyperbolic ratio test \cite{bc_root}, the series  $Z^{-\mathcal{V}} E_{\mathcal{V},C}(Z)$ absolutely hyperbolic convergent for all $Z\in\mathbb{BC}$.
Thus, the proof is complete.
\end{proof}
\begin{remark}
Under specific parameter values, distinct cases can be delineated as follows:
\begin{enumerate}
    \item[(i)]  For $\mathcal{V}=0,C=1$ we obtain $ E_{0,1}(Z)=\exp(Z)$.
    \item[(ii)]  For $\mathcal{V}=1$, we get $E_{1,C}=\frac{1}{C}\left[\exp(CZ)-1\right]$.
    \item[(iii)] When $\mathcal{V}=2$, we obtain $E_{2,C}=\frac{1}{C^2}\left[\exp(CZ)-1-CZ\right]$.
    \item[(iv)] For  $\mathcal{V}=0,C=j$ we have, $E_{0,j}=\cos Z+j\sin Z$.
    \item[(v)] When $C=0$, we obtain $E_{\mathcal{V},0}=\frac{Z^\mathcal{V}}{\Gamma(\mathcal{V}+1)}$.
    \item [(vi)] For $\mathcal{V}=\frac{1}{2},C=1, Z=w^2_1$, we get $E_{\frac{1}{2},1}(w^2_1)=\frac{2}{\sqrt{\pi}}e^{w^2_1}$ Erf $w_1$,  where Erf $w_1$ is error function.
\end{enumerate}
\end{remark}
\begin{theo}
  Let  $Z=z_1e_1+z_2e_2, \mathcal{V}=\nu_1e_1+\nu_2e_2=\alpha +j\beta, C=c_1e_1+c_2e_2 \in{\mathbb{BC}}$, where $e_1=\frac{1+k}{2}$ and $e_2=\frac{1-k}{2}$. Then the bicomplex Miller Ross function satisfies the following recurrence relations:
  \begin{enumerate}
         \item[(i)] $ E_{\mathcal{V},C}(Z)=C^{-\mathcal{V}}E_{0,C}(Z)$ for $\alpha=-\frac{l_1+l_2}{2}$ and $\beta=\frac{l_1-l_2}{2i}$ where $l_1,l_2\in \mathbb{N}\cup \{0\}$;
         \item[(ii)] $Z^3E_{\mathcal{V},C}(Z)=CZ^3E_{\mathcal{V}+1,C}(Z)+(\mathcal{V}+1)(\mathcal{V}+2)(\mathcal{V}+3)\left[E_{\mathcal{V}+3,C}(Z)-CE_{\mathcal{V}+4,C}(Z)\right]$;
         \item[(iii)] $Z^2E_{\mathcal{V},C}(Z)=\frac{Z^{\mathcal{V}+2}}{\Gamma(\mathcal{V}+1)}+C^2Z^2E_{\mathcal{V}+2,C}(Z)+C(\mathcal{V}+2)(\mathcal{V}+3)\left[E_{\mathcal{V}+3,C}(Z)-CE_{\mathcal{V}+4,C}(Z)\right]$;
         \item[(iv)]  $Z^2 E_{\mathcal{V},C}(Z)=\frac{Z^{\mathcal{V}+2}}{\Gamma(\mathcal{V}+1)}+\frac{CZ^{\mathcal{V}+3}}{\Gamma(\mathcal{V}+2)}+C^3Z^2E_{\mathcal{V}+3,C}(Z)+C^2(\mathcal{V}+3)(\mathcal{V}+4)\left[E_{\mathcal{V}+4,C}(Z)-CE_{\mathcal{V}+5,C}(Z)\right]$.
      \end{enumerate}
\end{theo}
\begin{proof}
(i) Since, $\alpha=-\frac{l_1+l_2}{2}$ and $\beta=\frac{l_1-l_2}{2i}$, idempotent components of $\mathcal{V}$ are given by $\nu_1=-l_1$ and $\nu_2=-l_2$ where $l_1,l_2\in \mathbb{N}\cup \{0\}$. From series representation of the bicomplex Miller Ross function \eqref{eq:gh}, we obtain
 \begin{align*}
     E_{\mathcal{V},C}(Z)&=Z^\mathcal{V}
  \sum_{r=0}^{\infty}\frac{(CZ)^r}{\Gamma{(\mathcal{V}+r+1)}}\\
  &={z_1}^{-l_1}\sum_{r=0}^{\infty}\frac{(c_1z_1)^r}{\Gamma{(-l_1+r+1)}}e_1+
  {z_2}^{-l_2}\sum_{r=0}^{\infty}\frac{(c_1z_1)^r}{\Gamma{(-l_2+r+1)}}e_2\\
  &={z_1}^{-l_1}\sum_{r=l_1}^{\infty}\frac{(c_1z_1)^r}{\Gamma{(-l_1+r+1)}}e_1+
  {z_2}^{-l_2}\sum_{r=l_2}^{\infty}\frac{(c_1z_1)^r}{\Gamma{(-l_2+r+1)}}e_2\\
  &={z_1}^{-l_1}\sum_{r=0}^{\infty}\frac{(c_1z_1)^{r+l_1}}{\Gamma{(r+1)}}e_1+
  {z_2}^{-l_2}\sum_{r=0}^{\infty}\frac{(c_1z_1)^{r+l_2}}{\Gamma{(r+1)}}e_2\;\;\left(\mbox{since}, \frac{1}{\Gamma(r)}=0\; \mbox{for}\; r\in{\mathbb{Z}-\mathbb{N}}\right)\\
  &={c_1}^{l_1}\sum_{r=0}^{\infty}\frac{(c_1z_1)^r}{\Gamma{(r+1)}}e_1+
  {c_2}^{l_2}\sum_{r=0}^{\infty}\frac{(c_1z_1)^r}{\Gamma{(r+1)}}e_2\\
  &=(c_1^{l_1}e_1+c_2^{l_2}e_2)\left[\sum_{r=0}^{\infty}\frac{(c_1z_1)^r}{\Gamma{(r+1)}}e_1+\sum_{r=0}^{\infty}\frac{(c_1z_1)^r}{\Gamma{(r+1)}}e_2\right]\\
  &=C^{-\mathcal{V}}E_{0,C}(Z).
 \end{align*}
 Thus, the proof of (i) is complete. \\
(ii) From the definition of Miller-Ross function \eqref{eq:5}, we have
     \begin{align}\label{eq:sl}
         z_1^3E_{\nu_1,c_1}(z_1)&={z_1}^{\nu_1+3}\sum_{r=0}^{\infty}\frac{(c_1z_1)^r}{\Gamma{(\nu_1+r+1)}}\nonumber\\
         &={z_1}^{\nu_1+3}\left[\frac{(\nu_1+1)(\nu_1+2)(\nu_1+3)}{\Gamma(\nu_1+4)}+\sum_{r=1}^{\infty}\frac{(c_1z_1)^n}{\Gamma{(\nu_1+r+1)}}\right]\nonumber\\
         &=(\nu_1+1)(\nu_1+2)(\nu_1+3){z_1}^{\nu_1+3}\left[\sum_{r=0}^{\infty}\frac{(c_1z_1)^r}{\Gamma{(\nu_1+r+4)}}-\sum_{r=0}^{\infty}\frac{(c_1z_1)^{r+1}}{\Gamma{(\nu_1+r+5)}}\right]\nonumber\\
         &\quad+{z_1}^{\nu_1+3}\sum_{r=0}^{\infty}\frac{(c_1z_1)^{r+1}}{\Gamma{(\nu_1+r+2)}}\nonumber\\
         &=(\nu_1+1)(\nu_1+2)(\nu_1+3)\left[E_{\nu_1+3,c_1}(z_1)-c_1E_{\nu_1+4,c_1}(z_1)\right]+c_1z_1^3E_{\nu_1+1,c_1}(z_1).
     \end{align}
     Similarly, we obtain
     \begin{align}\label{eq:co}
          z_2^3E_{\nu_2,c_2}(z_2)=(\nu_2+1)(\nu_2+2)(\nu_2+3)\left[E_{\nu_2+3,c_2}(z_2)-c_2E_{\nu_2+4,c_2}(z_2)\right]+c_2z_2^3E_{\nu_2+1,c_2}(z_2).
     \end{align}
     Using \eqref{eq:fa}, \eqref{eq:sl} and \eqref{eq:co}, we get
     \begin{align*}
         Z^3E_{\mathcal{V},C}(Z)&= z_1^3E_{\nu_1,c_1}(z_1)e_1+ z_2^3E_{\nu_2,c_2}(z_2)e_2\\
         &=\left[(\nu_1+1)(\nu_1+2)(\nu_1+3)\left[E_{\nu_1+3,c_1}(z_1)-c_1E_{\nu_1+4,c_1}(z_1)\right]+c_1z_1^3E_{\nu_1+1,c_1}(z_1)\right]e_1\\
         &\quad+\left[(\nu_2+1)(\nu_2+2)(\nu_2+3)\left[E_{\nu_2+3,c_2}(z_2)-c_2E_{\nu_2+4,c_2}(z_2)\right]+c_2z_2^3E_{\nu_2+1,c_2}(z_2)\right]e_2\\
         &=(\mathcal{V}+1)(\mathcal{V}+2)(\mathcal{V}+3)\left[E_{\mathcal{V}+3,C}(Z)-CE_{\mathcal{V}+4,C}(Z)\right]+CZ^3E_{\mathcal{V}+1,C}(Z).
    \end{align*}
    Hence, the proof of (ii) is complete. The relations (iii) and (iv) of this theorem can be proved by the similar method as used in the proof of (ii).
 \end{proof}
     \section{Integral representations of the bicomplex Miller Ross function}
 The integral representation of special functions unveils their deep connections to various mathematical domains, from physics to number theory. Through integral formulations, these functions become versatile tools in solving complex problems, offering insights into intricate phenomena. Their significance lies not only in computation but also in providing elegant solutions and profound insights into the underlying mathematical structures. In this section, we establish different type of integral representations of the bicomplex Miller-Ross function.
 \begin{theo}\label{th:ad}
    Assume that $Z=z_1e_1+z_2e_2, \mathcal{V}=\nu_1e_1+\nu_2e_2=\alpha +j\beta, C=c_1e_1+c_2e_2 \in{\mathbb{BC}}$, where $e_1=\frac{1+k}{2}$  and  $e_2=\frac{1-k}{2}$ with satisfies $\Re(\alpha)+1>|\Im(\beta)|$, the integral representation of bicomplex Miller-Ross function is given by
    \begin{align*}
        E_{\mathcal{V},C}(Z)=\frac{Z^{\mathcal{V}}}{\Gamma(\mathcal{V}+1)}\left[1+\sum_{r=1}^{\infty}\frac{1}{\Gamma(r)}\int_{D}\frac{(CZt)^r(1-t)^{\mathcal{V}}}{t}dt\right],
    \end{align*}
    where $D=(D_1(t_1),D_2(t_2))$ be the curve in $ \mathbb{BC}$ and $t=t_1e_1+t_2e_2\in\mathbb{BC}$ with $0\leq t_1,t_2\leq 1$.
 \end{theo}
 \begin{proof}
    From idempotent representation \eqref{eq:fa} of bicomplex Miller-Ross function, we have
    \begin{align*}
   &E_{\mathcal{V},C}(Z)\\&= E_{\nu_1,c_1}(z_1)e_1+E_{\nu_2,c_2}(z_2)e_2\\&={z_1}^{\nu_1}\left[\frac{1}{\Gamma(\nu_1+1)}+\sum_{r=1}^{\infty}\frac{(c_1z_1)^r}{\Gamma{(\nu_1+r+1)}}\right]e_1+{z_2}^{\nu_2}\left[\frac{1}{\Gamma(\nu_2+1)}+\sum_{r=1}^{\infty}\frac{(c_2z_2)^r}{\Gamma{(\nu_2+r+1)}}\right]e_2\\&={z_1}^{\nu_1}\left[\frac{1}{\Gamma(\nu_1+1)}+\sum_{r=1}^{\infty}\frac{\beta(\nu_1+1,r)(c_1z_1)^r}{\Gamma(\nu_1+1)\Gamma(r)}\right]e_1+{z_2}^{\nu_2}\left[\frac{1}{\Gamma(\nu_2+1)}+\sum_{r=1}^{\infty}\frac{\beta(\nu_2+1,r)(c_2z_2)^r}{\Gamma(\nu_2+1)\Gamma(r)}\right]e_2\\&=\frac{z_1^{\nu_1}}{\Gamma(\nu_1+1)}\left[1+\sum_{r=1}^{\infty}\frac{(c_1z_1)^r}{\Gamma(r)}\int_{0}^{1}t_1^{r-1}(1-t_1)^{\nu_1}dt_1\right]e_1\\&\quad+\frac{z_2^{\nu_2}}{\Gamma(\nu_2+1)}\left[1+\sum_{r=1}^{\infty}\frac{(c_2z_2)^r}{\Gamma(r)}\int_{0}^{1}t_2^{n-1}(1-t_2)^{\nu_2}dt_2\right]e_2\\
   &=\frac{z_1^{\nu_1}}{\Gamma(\nu_1+1)}\left[1+\sum_{r=1}^{\infty}\frac{1}{\Gamma(r)}\int_{D_1}\frac{(c_1z_1t_1)^{r}(1-t_1)^{\nu_1}}{t_1}dt_1\right]e_1\\&\quad+\frac{z_2^{\nu_2}}{\Gamma(\nu_2+1)}\left[1+\sum_{r=1}^{\infty}\frac{1}{\Gamma(r)}\int_{D_2}\frac{(c_2z_2t_2)^{r}(1-t_2)^{\nu_2}}{t_2}dt_2\right]e_2\\
   &=\frac{Z^{\mathcal{V}}}{\Gamma(\mathcal{V}+1)}\left[1+\sum_{r=1}^{\infty}\frac{1}{\Gamma(r)}\int_{D}\frac{(CZt)^r(1-t)^{\mathcal{V}}}{t}dt\right].
\end{align*}
Hence, the proof is completed.
 \end{proof}
 \begin{corollary}\label{co:s1}
     Setting $\mathcal{V}=0,C=1$ in Theorem \ref{th:ad}, we obtain the integral representation of $\exp(Z)$ as follows
     \begin{align*}
         \exp(Z)=1+\sum_{r=1}^{\infty}\frac{1}{\Gamma(r)}\int_{D}\frac{(Zt)^r}{t}dt,
     \end{align*}
     where $D=(D_1(t_1),D_2(t_2))$ is the curve in $ \mathbb{BC}$ and $t=t_1e_1+t_2e_2\in\mathbb{BC}$ with $0\leq t_1,t_2\leq 1$.
 \end{corollary}
 \begin{remark}
     Substituting $Z=w_1\in\mathbb{C}$ in Corollary \ref{co:s1}, we get
     \begin{align*}
         \exp(w_1)&=1+\sum_{r=1}^{\infty}\frac{1}{\Gamma(r)}\int_{D}\frac{(w_1t)^r}{t}dt\\
         &=1+\sum_{r=1}^{\infty}\frac{1}{\Gamma(r)}\left[\int_{0}^1\frac{(w_1t_1)^r}{t_1}dt_1e_1+\int_{0}^1\frac{(w_1t_2)^r}{t_2}dt_2e_2\right]\\
         &=1+\sum_{r=1}^{\infty}\frac{1}{\Gamma(r)}\left[\int_{0}^1\frac{(w_1u)^r}{u}du(e_1+e_2)\right]\\
         &=1+\left[\int_0^1\sum_{n=1}^{\infty}\frac{(w_1u)^n}{\Gamma(n)u}du\right]\\
         &=1+w_1\int_0^1e^{w_1u}du.
     \end{align*}
     In (\cite{sp_function}, chapter 7) the authors discussed the integral representation of confluent hypergeometric function ${}_1F_1(a;b;z)$. Setting $a=1$ and $b=2$ in (\cite{sp_function}, page 124, equation (9)), we get the integral representation of $\exp(w_1)$, which is same as our result.
 \end{remark}
 \begin{corollary}\label{co:t2}
     Setting  $\mathcal{V}=0,C=j$ in Theorem \ref{th:ad}, we have integral representation of
      $\mathbb{E}_{0,j}=\cos Z+j \sin Z$, given by
      \begin{align}
         \cos Z+j \sin Z=1+ \sum_{r=1}^{\infty}\frac{(j)^r}{\Gamma(r)}\int_{D}\frac{(Zt)^r}{t}dt,
      \end{align}
      where $D=(D_1(t_1),D_2(t_2))$ be the curve in $ \mathbb{BC}$ and $t=t_1e_1+t_2e_2\in\mathbb{BC}$ with $0\leq t_1,t_2\leq 1$.
       Now, equating both sides of the equation \eqref{co:t2}, we have
       \begin{align*}
           &\cos Z=1+ \sum_{r=1}^{\infty}\frac{(-1)^r}{\Gamma(2r)}\int_{D}\frac{(Zt)^{2r}}{t}dt,\\
           &\sin Z= \sum_{r=1}^{\infty}\frac{(-1)^{r-1}}{\Gamma(2r-1)}\int_{D}\frac{(Zt)^{2r-1}}{t}dt.
       \end{align*}
 \end{corollary}
 \begin{remark}
     Setting $Z=w_1\in\mathbb{C}$ in Corollary \ref{co:t2}, we observed that $$\cos w_1=1+ \sum_{r=1}^{\infty}\frac{(-1)^r}{\Gamma(2r)}\int_{D}\frac{(w_1t)^{2r}}{t}dt \;\;\mbox{and}\;\; \sin w_1=\sum_{r=1}^{\infty}\frac{(-1)^{r-1}}{\Gamma(2r-1)}\int_{D}\frac{(w_1t)^{2r-1}}{t}dt.$$ Similarly setting $a=1,b=2$ and $z=iw_1$ in the integral representation of ${}_1F_1(a;b;z)$  (\cite{sp_function}, page 124, equation (9)), we obtain integral representation of $\cos w_1$ and $\sin w_1$, which coincides with the result discussed in our context.
 \end{remark}
 \begin{theo}\label{th:mn}
     Assume that $Z=z_1e_1+z_2e_2, \mathcal{V}=\nu_1e_1+\nu_2e_2=\alpha_1 +j\alpha_2, M=\mu_1e_1+\mu_2e_2=\beta_1+j\beta_2, C=c_1e_1+c_2e_2 \in{\mathbb{BC}}$, where $e_1=\frac{1+k}{2}, e_2=\frac{1-k}{2}$ and satisfies the conditions $\Re(\alpha_1)>|\Im(\alpha_2)|,\Re(\beta_1)>|\Im(\beta_2)|$, the integral representation of bicomplex Miller-Ross function is given by
     \begin{align*}
         E_{\mathcal{V}+M,C}(Z)=\frac{Z^{\mathcal{V}+M}}{\Gamma(\mathcal{V}\Gamma(M)}\sum_{r=0}^{\infty}\frac{(CZ)^r}{r!}\int_A\int_Ds^{\mathcal{V}-1}(1-s)^{M+r}t^{M-1}(1-t)^rdsdt,
     \end{align*}
     where $A=(A_1(s_1),A_2(s_2))$ and $D=(D_1(t_1),D_2(t_2))$ are two curves in $\mathbb{BC}$ with $0\leq s_1,s_2,t_1,t_2\leq1$.
 \end{theo}
 \begin{proof}
   Using \eqref{eq:gh}, we obtain
   \begin{align}\label{eq:zm}
        &E_{\mathcal{V}+M,C}(Z)\\&=Z^{\mathcal{V}+M}\sum_{r=0}^{\infty}\frac{(CZ)^n}{\Gamma{(\mathcal{V}+M+r+1)}}\nonumber\\
        &={z_1}^{\nu_1+\mu_1}\sum_{r=0}^{\infty}\frac{(c_1z_1)^r}{\Gamma{(\nu_1+\mu_1+r+1)}}e_1+{z_2}^{\nu_2+\mu_2}\sum_{r=0}^{\infty}\frac{(c_2z_2)^r}{\Gamma{(\nu_2+\mu_2+r+1)}}e_2\nonumber\\
        &=\frac{{z_1}^{\nu_1+\mu_1}}{\Gamma(\nu_1)\Gamma(\mu_1)}\sum_{r=0}^{\infty}\frac{\Gamma(\nu_1)\Gamma(\mu_1+r+1)}{\Gamma{(\nu_1+\mu_1+r+1)}}\frac{\Gamma(\mu_1)\Gamma(r+1)}{\Gamma(\mu_1+r+1)}\frac{(c_1z_1)^r}{r!}e_1\nonumber\\
        &\quad+\frac{{z_2}^{\nu_2+\mu_2}}{\Gamma(\nu_2)\Gamma(\mu_2)}\sum_{r=0}^{\infty}\frac{\Gamma(\nu_2)\Gamma(\mu_2+r+1)}{\Gamma{(\nu_2+\mu_2+r+1)}}\cdot\frac{\Gamma(\mu_2)\Gamma(r+1)}{\Gamma(\mu_2+r+1)}\frac{(c_2z_2)^r}{r!}e_2\nonumber\\
        &=\frac{{z_1}^{\nu_1+\mu_1}}{\Gamma(\nu_1)\Gamma(\mu_1)}\sum_{r=0}^{\infty}B(\nu_1,\mu_1+r+1)B(\mu_1,r+1)\frac{(c_1z_1)^r}{r!}e_1\nonumber\\
        &\quad+\frac{{z_2}^{\nu_2+\mu_2}}{\Gamma(\nu_2)\Gamma(\mu_2)}\sum_{r=0}^{\infty}B(\nu_2,\mu_2+r+1)B(\mu_2,r+1)\frac{(c_2z_2)^r}{r!}e_2\nonumber\\
       & =\frac{{z_1}^{\nu_1+\mu_1}}{\Gamma(\nu_1)\Gamma(\mu_1)}\sum_{r=0}^{\infty}\int_0^1s_1^{\nu_1-1}(1-s_1)^{\mu_1+r}ds_1\int_0^1t_1^{\mu_1-1}(1-t_1)^rdt_1\frac{(c_1z_1)^r}{r!}e_1\nonumber\\
       &\quad+\frac{{z_2}^{\nu_2+\mu_2}}{\Gamma(\nu_2)\Gamma(\mu_2)}\sum_{r=0}^{\infty}\int_0^1s_2^{\nu_2-1}(1-s_2)^{\mu_2+r}ds_2\int_0^1t_2^{\mu_2-1}(1-t_2)^rdt_2\frac{(c_2z_2)^r}{r!}e_2.
   \end{align}
Let us consider two curves $A$ and $D$ in $\mathbb{BC}$, whose parametric representations are $A=(A_1(s_1),A_2(s_2))$ and $D=(D_1(t_1),D_2(t_2))$ respectively and $s=s_1e_1+s_2e_2, t=t_1e_1+t_2e_2\in{\mathbb{BC}}$ which satisfies the condition $0\leq s_1,s_2,t_1,t_2\leq1$. Now from equation \eqref{eq:zm}, we obtain
\begin{align*}
    & E_{\mathcal{V}+M,C}(Z)\\&= \left[\frac{{z_1}^{\nu_1+\mu_1}}{\Gamma(\nu_1)\Gamma(\mu_1)}e_1+\frac{{z_2}^{\nu_2+\mu_2}}{\Gamma(\nu_2)\Gamma(\mu_2)}e_2\right]\\&\quad\times\sum_{r=0}^{\infty}\int_A(s_1e_1+s_2e_2)^{\nu_1e_1+\nu_2e_2-1}(1-s_1e_1-s_2e_2)^{\mu_1e_1+\mu_2e_2+r}d(s_1e_1+s_2e_2)\\
     &\hspace{40pt}\cdot\int_D(t_1e_1+t_2e_2)^{\mu_1e_1+\mu_2e_2-1}(1-t_1e_1-t_2e_2)^rd(t_1e_1+t_2e_2)\left[\frac{(c_1z_1)^r}{r!}e_1+\frac{(c_2z_2)^r}{r!}e_2\right]\\
     &=\frac{Z^{\mathcal{V}+M}}{\Gamma(\mathcal{V})\Gamma(M)}\sum_{r=0}^{\infty}\frac{(CZ)^r}{r!}\int_A\int_Ds^{\mathcal{V}-1}(1-s)^{M+r}t^{M-1}(1-t)^rdsdt,
\end{align*}
 which completes the proof of the theorem.
 \end{proof}
 \begin{corollary}
     Substituting of $\mathcal{V}=1$ and $M=1$ in Theorem\ref{th:mn}, we have
     \begin{align*}
         \exp(CZ)=1+CZ+\sum_{r=0}^{\infty}\frac{(CZ)^{r+2}}{r!}\int_A\int_D(1-s)^{r+1}(1-t)^rdsdt.
     \end{align*}
 \end{corollary}
  \begin{theo}\label{th:w9}
      Let  $Z=z_1e_1+z_2e_2, \mathcal{V}=\nu_1e_1+\nu_2e_2=\alpha +j\beta, C=c_1e_1+c_2e_2 \in{\mathbb{BC}}$ with $\Re(\alpha)+1>|\Im(\beta)|$, then  bicomplex Miller-Ross function can be expresses as
     \begin{align*}
         E_{\mathcal{V},C}(Z)=Z^\mathcal{V}
  \sum_{r=0}^{\infty}\frac{(CZ)^r}{\int_{\gamma}e^{-t}t^{\mathcal{V}+r}dt},
     \end{align*}
    where $t=t_1e_1+t_2e_2\in\mathbb{BC}$ and $\gamma=(\gamma_1,\gamma_2)$ be a curve in $ \mathbb{BC}$, whose parametric equation is $\gamma(t)=(\gamma(t_1),\gamma(t_2))$ with $0<t_1,t_2<\infty$.
 \end{theo}
 \begin{proof}
 Using \eqref{eq:fa} and integral representation of Gamma function (\cite{sp_function}, chapter 2), we get
\begin{align*}
   E_{\mathcal{V},C}(Z)&= E_{\nu_1,c_1}(z_1)e_1+E_{\nu_2,c_2}(z_2)e_2\\&={z_1}^{\nu_1}\sum_{r=0}^{\infty}\frac{(c_1z_1)^r}{\Gamma{(\nu_1+r+1)}}e_1+{z_2}^{\nu_2}\sum_{r=0}^{\infty}\frac{(c_2z_2)^r}{\Gamma{(\nu_2+r+1)}}e_2\\&={z_1}^{\nu_1}\sum_{r=0}^{\infty}\frac{(c_1z_1)^r}{\int_0^{\infty}e^{-t_1}{t_1}^{\nu_1+r}dt_1}e_1+{z_2}^{\nu_2}\sum_{r=0}^{\infty}\frac{(c_2z_2)^re_2}{\int_0^{\infty}e^{-t_2}{t_2}^{\nu_2+r}dt_2}e_2\\&=(z_1^{\nu_1}e_1+z_2^{\nu_2}e_2)\left[\sum_{r=0}^{\infty}\frac{(c_1z_1)^re_1+(c_2z_2)^re_2}{\int_0^{\infty}e^{-t_1}{t_1}^{\nu_1+r}dt_1e_1+\int_0^{\infty}e^{-t_2}{t_2}^{\nu_2+r}dt_2e_2}\right]\\&=Z^\mathcal{V}
  \sum_{r=0}^{\infty}\frac{(CZ)^r}{\int_{\gamma}e^{-t}t^{\mathcal{V}+r}dt},
\end{align*}
 which completes the proof of the theorem.
 \end{proof}
 \begin{corollary}
     In Theorem \ref{th:w9}, setting $\mathcal{V}=\frac{1}{2},C=1, Z=w^2_1$, we obtain integral representation of error function $Erf$, which is given by
     \begin{align*}
         Erf w_1=\frac{\sqrt{\pi}}{2}e^{-w^2_{1}}\sum_{r=0}^{\infty}\frac{w^{2r+1}_1}{\int_{\gamma}e^{-t}t^{r+\frac{1}{2}}dt}.
     \end{align*}
 \end{corollary}
 \begin{theo}
     Suppose that  $Z=z_1e_1+z_2e_2, \mathcal{V}=\nu_1e_1+\nu_2e_2=\alpha +j\beta, C=c_1e_1+c_2e_2, s=se_1+s_2e_2 \in{\mathbb{BC}}$ with $\Re(\alpha)+1>|\Im(\beta)|,\Re(c_1z_1)<0$ and $\Re(c_2z_2)<0$, then
     \begin{align}\label{eq:h1}
      \int_\Omega\frac{(-C)^sZ^{\mathcal{V}+s}\Gamma(-s)\Gamma(1+s)}{\Gamma(\mathcal{V}+1+s)}ds=2{\pi} i E_{\mathcal{V},C}(Z),
     \end{align}
     where $\Omega=(\Omega_1,\Omega_2)$ is a path in $\mathbb{BC}$ whose components $\Omega_1$ and $\Omega_2$ are Barnes path starts at $-i\infty$ and runs to $+i\infty$ in $\mathbb{C}(i)$.
 \end{theo}
 \begin{proof}
     Setting $q=1, a_1=1, b_1=\nu_1+1$ and $z=c_1z_1$ in Theorem\ref{th:zd}, we obtain
     \begin{align}\label{eq:1i}
         \int_{\Omega_1}\frac{(-c_1z_1)^{s_1}\Gamma(-s_1)\Gamma(1+s_1)}{\Gamma(\nu_1+1+s_1)}ds_1&=\frac{2\pi i}{\Gamma(\nu_1+1)}{}_{1} F_{1}\left[\begin{matrix}&1;\\&\nu_1+1;&\end{matrix}c_1z_1\right]\nonumber\\
         &=\frac{2\pi i}{\Gamma(\nu_1+1)}\sum_{r=0}^{\infty}\frac{(1)_r(c_1z_1)^r}{(\nu_1+1)_r r!}\nonumber\\
         &=2\pi i\sum_{r=0}^{\infty}\frac{(c_1z_1)^r}{\Gamma(\nu_1+r+1)},
     \end{align}
     where $\Re(c_1z_1)<0$ and $\Omega_1$ is a Barnes path in $\mathbb{C}(i)$ starts at $-i\infty$ and runs to $+i\infty$ in a way that the poles $s_1=r\;(r\in\mathbb{N}_0)$ of $\Gamma(-s_1)$ lie to the right of $\Omega_1$ and the poles $s_1=-l\;(l\in\mathbb{N})$ of $\Gamma(1+s_1)$ to the right of it.
     Multiplying both side by $z_1^{\nu_1}$ in \eqref{eq:1i}, we get
     \begin{align}\label{eq:vj}
        \int_{\Omega_1}\frac{(-c_1)^{s_1}z_1^{{\nu_1}+s_1}\Gamma(-s_1)\Gamma(1+s_1)}{\Gamma(\nu_1+1+s_1)}ds_1=2\pi i\; E_{\nu_1,c_1}(z_1),\;\Re(c_1z_1)<0.
     \end{align}
      Similarly, after putting the value of $q=1, a_1=1, b_1=\nu_2+1$ and $z=c_2z_2$ in Theorem\ref{th:zd}, we have
     \begin{align}\label{eq:xi}
         \int_{\Omega_2}\frac{(-c_2)^{s_2}z_2^{\nu_2+s_2}\Gamma(-s_2)\Gamma(1+s_2)}{\Gamma(\nu_2+1+s_2)}ds_2=2\pi i\; E_{\nu_2,c_2}(z_2),
     \end{align}
      where $\Re(c_2z_2)<0$ and $\Omega_2$ is a Barnes path in $\mathbb{C}(i)$ running from $-i\infty$ to $+i\infty$ in a way that the poles $s_2=r\;(r\in\mathbb{N}_0)$ of $\Gamma(-s_2)$ lie to the right of $\Omega_2$ and the poles $s_2=-l\;(l\in\mathbb{N})$ of $\Gamma(1+s_2)$ to the right of it. Using \eqref{eq:fa}, \eqref{eq:vj} and \eqref{eq:xi}, we get
      \begin{align*}
          &E_{\mathcal{V},C}(Z)\\
          &=E_{\nu_1,c_1}(z_1)e_1+E_{\nu_2,c_2}(z_2)e_2\\
          &=\frac{1}{2\pi i}\int_{\Omega_1}\frac{(-c_1)^{s_1}z_1^{{\nu_1}+s_1}\Gamma(-s_1)\Gamma(1+s_1)}{\Gamma(\nu_1+1+s_1)}ds_1e_1+\frac{1}{2\pi i}\int_{\Omega_2}\frac{(-c_2)^{s_2}z_2^{\nu_2+s_2}\Gamma(-s_2)\Gamma(1+s_2)}{\Gamma(\nu_2+1+s_2)}ds_2e_2\\
          &=\frac{1}{2\pi i}\int_{(\Omega_1,\Omega_2)}\left[\frac{(-c_1)^{s_1}z_1^{{\nu_1}+s_1}\Gamma(-s_1)\Gamma(1+s_1)e_1+(-c_2)^{s_2}z_2^{\nu_2+s_2}\Gamma(-s_2)\Gamma(1+s_2)e_2}{\Gamma(\nu_1+1+s_1)e_1+\Gamma(\nu_2+1+s_2)e_2}\right]d(s_1e_1+s_2e_2)\\
          &= \frac{1}{2\pi i}\int_\Omega\frac{(-C)^sZ^{\mathcal{V}+s}\Gamma(-s)\Gamma(1+s)}{\Gamma(\mathcal{V}+1+s)}ds,\;\Re(c_1z_1)<0,\;\Re(c_2z_2)<0.
     \end{align*}
      Hence, the proof is completed.
 \end{proof}

 \section{Differential relations and differential equation of the bicomplex Miller Ross function}
 Differential relations of special functions elucidate fundamental properties and behaviors crucial in diverse scientific disciplines. These relations facilitate the understanding of dynamic systems, ranging from quantum mechanics to fluid dynamics, by providing insights into the interplay between different variables. Through differential relations, special functions serve as indispensable tools in modeling and analyzing complex phenomena, enriching our understanding of the natural world.
 In this section, we derive differential relations and establish differential equation which are satisfied by bicomplex Miller-Ross function.
 \begin{theo}\label{th:ui}
    Let  $Z=z_1e_1+z_2e_2, \mathcal{V}=\nu_1e_1+\nu_2e_2=\alpha +j\beta, C=c_1e_1+c_2e_2 \in{\mathbb{BC}}$ and $k\in\mathbb{N}$, then
   \begin{align*}
    \frac{d^k}{dZ^k}(E_{\mathcal{V},C}(Z))=\sum\limits_{p=1}^{k}\frac{C^{k-p}Z^{\mathcal{V}-p}}{\Gamma(\mathcal{V}-p+1)}+C^kE_{\mathcal{V},C}(Z).
    \end{align*}
\end{theo}
\begin{proof}
    We will prove the theorem using principle of mathematical induction. By using derivative formula for bicomplex function, we get
\begin{align*}
    \frac{d}{dZ} \left( E_{\mathcal{V},C}(Z)\right)
     &=\frac{d}{dz_1}\left(z_1^{\nu_1}
  \sum_{r=0}^{\infty}\frac{(c_1z_1)^r}{\Gamma{(\nu_1+r+1)}}\right)e_1+\frac{d}{dz_2}\left(z_2^{\nu_2}
  \sum_{r=0}^{\infty}\frac{(c_2z_2)^r}{\Gamma{(\nu_2+r+1)}}e_2\right) e_2\nonumber\\
     &=\left(
  \sum_{r=0}^{\infty}\frac{(r+\nu_1)c_1^rz_1^{r+\nu_1-1}}{\Gamma{(\nu_1+r+1)}}\right)e_1+\left(\sum_{r=0}^{\infty}\frac{(r+\nu_2)c_2^rz_2^{r+\nu_2-1}}{\Gamma{(\nu_2+r+1)}}\right) e_2\nonumber\\
     &=\left(\frac{z_1^{\nu_1-1}}{\Gamma(\nu_1)}+ c_1z_1^{\nu_1}\sum_{r=0}^{\infty}\frac{(c_1z_1)^r}{\Gamma{(\nu_1+r+1)}}\right)e_1+\left(\frac{z_2^{\nu_2-1}}{\Gamma(\nu_2)}+ c_2z_2^{\nu_2}\sum_{r=0}^{\infty}\frac{(c_2z_2)^r}{\Gamma{(\nu_2+r+1)}}\right) e_2\\
    &=\left(\frac{z_1^{\nu_1-1}}{\Gamma(\nu_1)}+ c_1E_{\nu_1,c_1}(z_1)\right)e_1+\left(\frac{z_2^{\nu_2-1}}{\Gamma(\nu_2)}+ c_2E_{\nu_2,c_2}(z_2)\right) e_2\\
    &=\left(\frac{z_1^{\nu_1-1}}{\Gamma(\nu_1)}e_1+\frac{z_2^{\nu_2-1}}{\Gamma(\nu_2)}e_2\right)+(c_1e_1+c_2e_2)\left(E_{\nu_1,c_1}(z_1)e_1+E_{\nu_2,c_2}(z_2)e_2\right)\\
    &=\frac{Z^{\mathcal{V}-1}}{\Gamma(\mathcal{V})}+CE_{\mathcal{V},C}(Z).
    \end{align*}
Therefore, the given statement is true for $k=1$. Assume that given statement is true for $k=m$ and it implies that
\begin{align*}
   \frac{d^m}{dZ^m}(E_{\mathcal{V},C}(Z))=\sum\limits_{p=1}^{m}\frac{C^{m-p}Z^{\mathcal{V}-p}}{\Gamma(\mathcal{V}-p+1)}+C^mE_{\mathcal{V},C}(Z).
\end{align*}
Now,
\begin{align*}
     &\frac{d^{m+1}}{dZ^{m+1}} \left( E_{\mathcal{V},C}(Z)\right)
     =\frac{d}{dZ}\left[ \frac{d^m}{dZ^m} \left(E_{\mathcal{V},C}(Z)\right) \right]\\
     &=\frac{d}{dz_1}\left(\sum\limits_{p=1}^{m}\frac{c_1^{m-p}z^{\nu_1-p}}{\Gamma(\nu_1-p+1)}+c_1^m E_{\nu_1,c_1}(z_1)\right)e_1+\frac{d}{dz_2}\left(\sum\limits_{p=1}^{m}\frac{c_2^{m-p}z^{\nu_2-p}}{\Gamma(\nu_2-p+1)}+c_2^m E_{\nu_2,c_2}(z_2)\right)e_2\\
     &=\left(\sum\limits_{p=1}^{m}\frac{c_1^{m-p}z^{\nu_1-p-1}}{\Gamma(\nu_1-p)}+\frac{c_1^mz_1^{\nu_1-p}}{\Gamma(\nu_1)}+c_1^{m+1}E_{\nu_1,c_1}(z_1)\right)e_1+\left(\sum\limits_{p=1}^{m}\frac{c_2^{m-p}z^{\nu_2-p-1}}{\Gamma(\nu_2-p)}+\frac{c_2^mz_2^{\nu_2-p}}{\Gamma(\nu_2)}+c_2^{m+1}E_{\nu_2,c_2}(z_2)\right)e_2\\
     &=\left(\sum\limits_{p=1}^{m+1}\frac{c_1^{m+1-p}z^{\nu_1-p}}{\Gamma(\nu_1-p+1)}+c_1^{m+1}E_{\nu_1,c_1}(z_1)\right)e_1+\left(\sum\limits_{p=1}^{m+1}\frac{c_2^{m+1-p}z^{\nu_2-p}}{\Gamma(\nu_2-p+1)}+c_2^{m+1}E_{\nu_2,c_2}(z_2)\right)e_2\\
     &=\sum\limits_{p=1}^{m+1}\frac{C^{m+1-p}Z^{\mathcal{V}-p}}{\Gamma(\mathcal{V}-p+1)}+C^{m+1}E_{\mathcal{V},C}(Z).
\end{align*}
Thus, given statement is true for $k=1$ and for $k=m+1$ is true, whenever the statement is true for $k=m$. Hence, by principle of mathematical induction, the statement is true for all $k\in{\mathbb{N}}$. Hence the theorem is proved.
\end{proof}
\begin{theo}\label{th:lk}
     Assume that $Z=z_1e_1+z_2e_2, \mathcal{V}=\nu_1e_1+\nu_2e_2=\alpha +j\beta, C=c_1e_1+c_2e_2, M=m_1e_1+m_2e_2 \in{\mathbb{BC}}$ where $m_1,m_2\in\mathbb{Z^+}$, then
     \begin{align*}
         \frac{d^k}{dZ^k}(E_{\mathcal{V}+M,C}(Z))=E_{\mathcal{V}+M-k,C}(Z).
     \end{align*}
\end{theo}
\begin{proof}
    Using power series representation \eqref{eq:gh} of bicomplex Miller-Ross function, we obtain
    \begin{align*}
        & \frac{d}{dZ} \left(E_{\mathcal{V}+M,C}(Z)\right)\\
         &=\frac{d}{dZ}\left(Z^{\mathcal{V}+M}
  \sum_{r=0}^{\infty}\frac{(CZ)^r}{\Gamma{(\mathcal{V}+M+r+1)}}\right)\\
     &=\frac{d}{dz_1}\left(z_1^{\nu_1+m_1}
  \sum_{r=0}^{\infty}\frac{(c_1z_1)^r}{\Gamma{(\nu_1+m_1+r+1)}}\right)e_1+\frac{d}{dz_2}\left(z_2^{\nu_2+m_2}
  \sum_{r=0}^{\infty}\frac{(c_2z_2)^r}{\Gamma{(\nu_2+m_2+r+1)}}e_2\right) e_2\\
     &=\left(
  \sum_{r=0}^{\infty}\frac{(\nu_1+m_1+r)c_1^rz_1^{\nu_1+m_1+r-1}}{\Gamma{(\nu_1+m_1+r+1)}}\right)e_1+\left(\sum_{r=0}^{\infty}\frac{(\nu_2+m_2+r)c_2^rz_2^{\nu_2+m_2+r-1}}{\Gamma{(\nu_2+m_2+r+1)}}\right) e_2\\
  &=\left(z_1^{\nu_1+m_1-1}
  \sum_{r=0}^{\infty}\frac{(c_1z_1)^r}{\Gamma{(\nu_1+m_1+r)}}\right)e_1+\left(z_2^{\nu_2+m_2-1}\sum_{r=0}^{\infty}\frac{(c_2z_2)^r}{\Gamma{(\nu_2+m_2+r)}}e_2\right) e_2\\
  &=\left(Z^{\mathcal{V}+M-1}
  \sum_{r=0}^{\infty}\frac{(CZ)^r}{\Gamma{(\mathcal{V}+M+r)}}\right)\\
  &=E_{\mathcal{V}+M-1}(Z).
    \end{align*}
    Therefore, the given statement is true for $k=1$. Assume that given statement is true for $k=m$.
    Now,
    \begin{align*}
        &\frac{d^{m+1}}{dZ^{m+1}}\left(E_{\mathcal{V}+M,C}(Z)\right)\\
        &=\frac{d}{dZ}(E_{\mathcal{V}+M-m,C}(Z))\\
         &=\frac{d}{dZ}\left(Z^{\mathcal{V}+M-m}
  \sum_{r=0}^{\infty}\frac{(CZ)^r}{\Gamma{(\mathcal{V}+M-m+r+1)}}\right)\\
     &=\frac{d}{dz_1}\left(
  \sum_{r=0}^{\infty}\frac{c_1^rz_1^{\nu_1+m_1-m+r}}{\Gamma{(\nu_1+m_1-m+r+1)}}\right)e_1+\frac{d}{dz_2}\left(
  \sum_{r=0}^{\infty}\frac{c_2^rz_2^{\nu_2+m_2-m+r}}{\Gamma{(\nu_2+m_2-m+r+1)}}e_2\right) e_2\\
     &=\left(
  \sum_{r=0}^{\infty}\frac{(\nu_1+m_1-m+r)c_1^rz_1^{\nu_1+m_1-m+r-1}}{\Gamma{(\nu_1+m_1-m+r+1)}}\right)e_1+\left(\sum_{r=0}^{\infty}\frac{(\nu_2+m_2-m+r)c_2^rz_2^{\nu_2+m_2-m+r-1}}{\Gamma{(\nu_2+m_2-m+r+1)}}\right) e_2\\
  &=\left(z_1^{\nu_1+m_1-m-1}
  \sum_{n=0}^{\infty}\frac{(c_1z_1)^r}{\Gamma{(\nu_1+m_1-m+r)}}\right)e_1+\left(z_2^{\nu_2+m_2-m-1}\sum_{r=0}^{\infty}\frac{(c_2z_2)^r}{\Gamma{(\nu_2+m_2-m+r)}}e_2\right) e_2\\
  &=\left(Z^{\mathcal{V}+M-m-1}
  \sum_{r=0}^{\infty}\frac{(CZ)^r}{\Gamma{(\mathcal{V}+M-m+r)}}\right)\\
  &=E_{\mathcal{V}+M-m-1,C}(Z),
    \end{align*}
    which shows that the given statement is true for $k=m+1$. Hence, by principle of mathematical induction, the given statement is true for all $k\in\mathbb{N}$.
\end{proof}
\begin{remark}
    Setting $M=0$ in Theorem \ref{th:lk} and using Theorem \ref{th:ui}, we obtain
    \begin{align*}
        E_{\mathcal{V}-k,C}(Z)=\sum\limits_{p=1}^{k}\frac{C^{k-p}Z^{\mathcal{V}-p}}{\Gamma(\mathcal{V}-p+1)}+C^k E_{\mathcal{V},C}(Z),
    \end{align*}
    where $Z=z_1e_1+z_2e_2, \mathcal{V}=\nu_1e_1+\nu_2e_2=\alpha +j\beta, C=c_1e_1+c_2e_2, \in{\mathbb{BC}}$ and $k\in\mathbb{Z^+}$.
\end{remark}
\begin{remark}
    Setting $M=1$ and $k=1$ in Theorem \ref{th:lk} and using Theorem \ref{th:ui}, we get
    \begin{align*}
            E_{\mathcal{V},C}(Z)=\frac{Z^\mathcal{V}}{\Gamma(\mathcal{V}+1)}+C E_{\mathcal{V}+1,C}(Z),
       \end{align*}
       where $Z=z_1e_1+z_2e_2, \mathcal{V}=\nu_1e_1+\nu_2e_2=\alpha +j\beta, C=c_1e_1+c_2e_2, \in{\mathbb{BC}}$.
\end{remark}
\begin{theo}
     Suppose that  $Z=z_1e_1+z_2e_2, \mathcal{V}=\nu_1e_1+\nu_2e_2=\alpha +j\beta, C=c_1e_1+c_2e_2 \in{\mathbb{BC}}$ with $\Re(\alpha)+1>|\Im(\beta)|$ and $Z\notin \mathbb{O}_2$. Bicomplex Miller Ross function $ E_{\mathcal{V},C}(Z)$ satisfies the differential equation $$Z\frac{d^2U}{dZ^2}+(1-\mathcal{V}-CZ)\frac{dU}{dZ}+(\mathcal{V}-1)CU=0.$$
 \end{theo}
 \begin{proof}
  Now, every bicomplex Miller-Ross function have another representation through bicomplex confluent hypergeometric function as
  \begin{align*}
     E_{\mathcal{V},C}(Z)
  &={z_1}^{\nu_1}\sum_{r=0}^{\infty}\frac{(c_1z_1)^r}{\Gamma{(\nu_1+r+1)}}e_1+{z_2}^{\nu_2}\sum_{r=0}^{\infty}\frac{(c_2z_2)^r}{\Gamma{(\nu_2+r+1)}}e_2\\
  &={z_1}^{\nu_1}\sum_{r=0}^{\infty}\frac{(c_1z_1)^r}{(\nu_1+1)_r\Gamma{(\nu_1+1)}}e_1+{z_2}^{\nu_2}\sum_{r=0}^{\infty}\frac{(c_2z_2)^r}{(\nu_2+1)_r\Gamma{(\nu_2+1)}}e_2\\
  &=\frac{z_1^{\nu_1}}{\Gamma(\nu_1+1)}\sum_{r=0}^{\infty}\frac{(1)_r(c_1z_1)^r}{(\nu_1+1)_rr!}e_1+\frac{z_2^{\nu_2}}{\Gamma(\nu_2+1)}\sum_{r=0}^{\infty}\frac{(1)_r(c_2z_2)^r}{(\nu_2+1)_rr!}e_2\\
 &=\frac{z_1^{\nu_1}}{\Gamma(\nu_1+1)}\times{}_{1} F_{1}\left[\begin{matrix} &1 ;\\& \nu_1+1;&\end{matrix}c_1z_1\right]e_1+\frac{z_2^{\nu_2}}{\Gamma(\nu_2+1)}\times{}_{1} F_{1}\left[\begin{matrix} &1 ;\\& \nu_2+1;&\end{matrix}c_2z_2\right]e_2\\
   &=\frac{Z^\mathcal{V}}{\Gamma(\mathcal{V}+1)}\times{}_{1} F_{1}\left[\begin{matrix} &1 ;\\& \mathcal{V}+1;&\end{matrix}CZ\right].
  \end{align*}
   By substituting $p=1$ and $q=1$ in the equation \eqref{eq:in}, we obtain bicomplex confluent hypergeometric function $W={}_{1} F_{1}\left[\begin{matrix} &a_1 ;\\& b_1;&\end{matrix}Y\right]$ satisfying the differential equation
   \begin{align}\label{eq:nk}
       \left[Y\frac{d^2}{dY^2}+(b_1-Y)\frac{d}{dY}-a_1\right]W=0.
   \end{align}
 Now putting, $a_1=1$, $b_1=\mathcal{V}+1$ and $Y=CZ$ in the equation \eqref{eq:nk}, we get following differential equation
 \begin{align}\label{eq:kd}
     \left[Z\frac{d^2}{dZ^2}+(\mathcal{V}+1-CZ)\frac{d}{dZ}-C\right]W=0,
 \end{align}
 whose one solution is $W={}_{1} F_{1}\left[\begin{matrix} &1 ;\\& \mathcal{V}+1;&\end{matrix}CZ\right]$. Since  $E_{\mathcal{V},C}(Z)=\frac{Z^\mathcal{V}}{\Gamma(\mathcal{V}+1)}W$, then if we put $W=Z^\mathcal{-V}\Gamma(\mathcal{V}+1)U$ in \eqref{eq:kd}, we have
 \begin{align*}
     &Z\left(Z^{-\mathcal{V}}\frac{d^2U}{dZ^2}-2\mathcal{V}Z^{-(\mathcal{V}+1)}\frac{dU}{dZ}+\mathcal{V}(\mathcal{V}+1)Z^{\mathcal{-V}-2}U\right)\\&\hspace{90pt}+(\mathcal{V}+1-CZ)\left(Z^\mathcal{-V}\frac{dU}{dZ}-\mathcal{V}Z^{-(\mathcal{V}+1)}U\right)-CZ^\mathcal{-V}U=0
 \end{align*}
 which implies that
 \begin{align*}
     Z\frac{d^2U}{dZ^2}+(1-\mathcal{V}-CZ)\frac{dU}{dZ}+(\mathcal{V}-1)CU=0.
 \end{align*}
 Hence, the proof is completed.
 \end{proof}
\section{Bicomplex holomorphicity and Taylor series representation of the bicomplex Miller Ross function}
In this section, we discuss the bicomplex holomorphicity of bicomplex Miller Ross function. Moreover, we establish Taylor series representation of that function.

\begin{theo}\label{th:w4}
     If $Z=w_1+jw_2=z_1e_1+z_2e_2, \mathcal{V}=\alpha +j\beta=\nu_1e_1+\nu_2e_2, C=c_1e_1+c_2e_2 \in{\mathbb{BC}}$, then for a fixed $Z\notin\mathbb{O}_2$, bicomplex Miller-Ross function $E_{\mathcal{V},C}(Z) $ is an analytic function of $\mathcal{V}$ and $C$.
\end{theo}
\begin{proof}
    The idempotent representation of the bicomplex Miller-Ross function is,
\begin{align}\label{eq:18}
    E_{\mathcal{V},C}(Z)=\mathbb{E}_{\nu_1,c_1}(z_1)e_1+E_{\nu_2,c_2}(z_2)e_2.
\end{align}
Setting $e_1=\frac{1+ij}{2}$ and $e_2=\frac{1-ij}{2}$ in \eqref{eq:18}, we get
\begin{align}\label{eq:19}
     E_{\mathcal{V},C}(Z)&=E_{\nu_1,c_1}(z_1)\left(\frac{1+ij}{2}\right)+E_{\nu_2,c_2}(z_2)\left(\frac{1-ij}{2}\right)\nonumber\\
     &=\frac{1}{2}(E_{\nu_1,c_1}(z_1)+E_{\nu_2,c_2}(z_2))+\frac{ij}{2}(E_{\nu_1,c_1}(z_1)-E_{\nu_2,c_2}(z_2))
\end{align}
Let us consider $ E_{\mathcal{V},C}(Z)=f_1(\alpha,\beta)+jf_2(\alpha,\beta)$. Then from the equation \eqref{eq:19}, we obtain
\begin{align*}
    &f_1(\alpha,\beta)=\frac{1}{2}(E_{\nu_1,c_1}(z_1)+E_{\nu_2,c_2}(z_2))\\
   &f_2(\alpha,\beta)= \frac{i}{2}(E_{\nu_1,c_1}(z_1)-E_{\nu_2,c_2}(z_2)).
\end{align*}
Now,
\begin{align*}
   & \frac{\partial f_1}{\partial \alpha}=\frac{1}{2}(E'_{\nu_1,c_1}(z_1)+E'_{\nu_2,c_2}(z_2)) \\
    & \frac{\partial f_1}{\partial \beta}=\frac{-i}{2}\left(E'_{\nu_1,c_1}(z_1)\right) + \frac{i}{2}\left(E'_{\nu_2,c_2}(z_2)\right)\\
    &\frac{\partial f_2}{\partial \alpha}=\frac{i}{2}\left(E'_{\nu_1,c_1}(z_1)\right) - \frac{i}{2}\left(E'_{\nu_2,c_2}(z_2)\right)\\
    &\frac{\partial f_2}{\partial \beta}=\frac{i}{2}\left(E'_{\nu_1,c_1}(z_1)\right)\cdot(-i) - \frac{i}{2}\left(E'_{\nu_2,c_2}(z_2)\right)\cdot (i_1)\\
    &\hspace{20pt}= \frac{1}{2}\left(E'_{\nu_1,c_1}(z_1)+E'_{\nu_2,c_2}(z_2)\right)
\end{align*}
from the above equation, we get
\begin{align*}
    \frac{\partial f_1}{\partial \alpha}= \frac{\partial f_2}{\partial \beta} \quad
      and \quad  \frac{\partial f_1}{\partial \beta}=- \frac{\partial f_2}{\partial \alpha}.
\end{align*}
Therefore, $f_1$ and $f_2$ satisfies bicomplex Cauchy-Riemann equation. Again, for fixed non zero value of  $z_1,z_2$ the complex Miller Ross function $E_{\nu_1,c_1}(z_1)$ and $E_{\nu_2,c_2}(z_2)$ are analytic function of $\nu_1$ and $\nu_2$  respectively, then for a fixed $Z\notin\mathbb{O}_2$, $f_1$ and $f_2$ are holomorphic functions of $\alpha$ and $\beta$. By using analyticity condition of the bicomplex function \eqref{eq:i5} we obtain, bicomplex Miller-Ross function $E_{\mathcal{V},C}(Z)$ is an analytic function of $\mathcal{V}$.
\end{proof}
\begin{theo}
     Suppose that $\mathcal{V}=\alpha +j\beta=\nu_1e_1+\nu_2e_2, C=c_1e_1+c_2e_2 \in{\mathbb{BC}}$ with $Z_0\in\mathbb{BC}-\mathbb{O}_2$ is a fixed bicomplex number, then Taylor series representation of bicomplex Miller Ross function is given by
     \begin{align*}
          E_{\mathcal{V},C}(Z_0)=A\sum_{k=0}^{\infty}\frac{(\log Z_0)^k}{k!}(\mathcal{V})^k ,
     \end{align*}
     where $A=\sum\limits_{r=0}^{\infty}\frac{(CZ_0)^r}{\Gamma{(\mathcal{V}+r+1)}}$.
\end{theo}
\begin{proof}
    Let us consider $Z=Z_0=z_{10}e_1+z_{20}e_2$ be a fixed non singular bicomplex number. Now using properties of bicomplex logarithmic function \cite{multicomplex} and \eqref{eq:gh}, we obtain
    \begin{align*}
         E_{\mathcal{V},C}(Z_0)&=Z_0^\mathcal{V}
  \sum_{r=0}^{\infty}\frac{(CZ_0)^r}{\Gamma{(\mathcal{V}+r+1)}}\\
  &=\exp\left(\log Z_0^\mathcal{V}\right)\sum_{r=0}^{\infty}\frac{(CZ_0)^r}{\Gamma{(\mathcal{V}+r+1)}}\\
  &=\exp\left(\log z_{10}^{\nu_1}e_1+\log z_{20}^{\nu_2}e_2\right)\sum_{r=0}^{\infty}\frac{(CZ_0)^r}{\Gamma{(\mathcal{V}+r+1)}}\\
  &=\left[\exp(\nu_1\log z_{10})e_1+\exp(\nu_2\log z_{20})e_2)\right]\sum_{r=0}^{\infty}\frac{(CZ_0)^r}{\Gamma{(\mathcal{V}+r+1)}}\\
  &=\left[\sum_{k=0}^{\infty}\frac{(\nu_1 \log z_{10})^k}{k!}e_1+\sum_{k=0}^{\infty}\frac{(\nu_2 \log z_{20})^k}{k!}e_2\right]\sum_{r=0}^{\infty}\frac{(CZ_0)^r}{\Gamma{(\mathcal{V}+r+1)}}\\
  &=\left[\sum_{k=0}^{\infty}\frac{(\nu_1e_1+\nu_2e_2)^k (\log z_{10}e_1+\log z_{20}e_2)^k}{k!}\right]\sum_{r=0}^{\infty}\frac{(CZ_0)^r}{\Gamma{(\mathcal{V}+r+1)}}\\
  &=\left[\sum_{k=0}^{\infty}\frac{(\log Z_0)^k}{k!}(\mathcal{V})^k\right]\sum_{r=0}^{\infty}\frac{(CZ_0)^r}{\Gamma{(\mathcal{V}+r+1)}}\\
  &=A\sum_{k=0}^{\infty}\frac{(\log Z_0)^k}{k!}(\mathcal{V})^k .
    \end{align*}
    Hence the proof is completed.
\end{proof}
\begin{theo}
    Assume that $Z=z_1e_1+z_2e_2, \mathcal{V}=\alpha +j\beta=\nu_1e_1+\nu_2e_2,C=c_1e_1+c_2e_2 \in{\mathbb{BC}}$, with satisfies the conditions $\alpha=\frac{l_1+l_2}{2}$ and $\beta=\frac{l_2-l_1}{2i}$ where $l_1,l_2\in \mathbb{Z}$, then bicomplex Miller Ross function $E_{\mathcal{V},C}(Z)$ is an entier function of $Z$.
\end{theo}
\begin{proof}
     Since $\alpha=\frac{l_1+l_2}{2}$ and $\beta=\frac{l_2-l_1}{2i}$, idempotent components of $\mathcal{V}$ are given by $\nu_1=l_1$ and $\nu_2=l_2$, where $l_1,l_2\in \mathbb{Z}$. Now from the series representation of Miller Ross function \eqref{eq:5}, we see that if $\nu\in\mathbb{N}\cup \{0\}$, then $E_{\nu,c}(z)$ is analytic function of $z\in\mathbb{C}$. For negative integer values of $\nu$ complex Miller Ross function can be written as $E_{\nu,c}(z)=c^{-\nu}E_{0,c}(z)$ and is also analytic function of $z\in\mathbb{C}$. Hence we can proceed to complete the proof using similar technique as the proof of the previous theorem.

\end{proof}
 \section{Bicomplex fractional differential equations and bicomplex generalized fractional kinetic equation}
 In \cite{frac-kinetic}, the authors discussed the standard form of the kinetic equation and extended it to a fractional kinetic equation, subsequently obtaining its solutions. The fractional generalization of the standard kinetic equation is given by
\begin{align}\label{eq:l9}
    N(t)-N_0=-c^\nu D^{-\nu}_tN(t),
\end{align}
where $\nu>0$ and $D^{-\nu}_t$ is the Riemann-Liouville fractional integral operator. Furthermore in \cite{gen-kinetic},  the authors explored generalizations of the fractional kinetic equation in terms of the Mittag-Leffler functions. In this section, we investigate the solution of new bicomplex generalization of the fractional kinetic equation \eqref{eq:l9} and established some fractional differential equations whose solutions is related to bicomplex Miller Ross function.\\
The Riemann-Liouville fractional integration of bicomplex order for function $f(t)=t^u$ is given by \cite{rlfo}:
 \begin{align}\label{eq:k5}
     D^{-M}_tt^u=\frac{\Gamma(u+1)}{\Gamma(u+M+1)}t^{u+M},
 \end{align}
 where $M=m_1+jm_2\in\mathbb{BC}$ with $\Re(m_1)>|\Im(m_2)|$ and $u>-1$. The bicomplex generalization of  \eqref{eq:l9} is given by replacing the Riemann-Liouville fractional operator $D^{-\nu}_t$ by the Riemann- Liouville fractional operator of bicomplex order $D^{-\mathcal{V}}_t$ in \eqref{eq:l9} as follows
  \begin{align*}
    N(t)-N_0=-c^\mathcal{V} D^{-\mathcal{V}}_tN(t),
     \end{align*}
     where $\mathcal{V}=\alpha +j\beta\in\mathbb{BC}$ with satisfies the condition $\Re(\alpha)>|\Im(\beta)|$.
     \begin{theo}\label{th:a2}
         Suppose that $\mathcal{V}=\alpha +j\beta \in{\mathbb{BC}}$, $c>0$ with satisfies the condition $\Re(\alpha)>|\Im(\beta)|$, then the solution of the bicomplex generalization of the fractional kinetic equation
    \begin{align}\label{eq:n7}
    N(t)-N_0=-c^\mathcal{V} D^{-\mathcal{V}}_tN(t),
     \end{align}
is given by
\begin{align*}
    N(t)=N_0\sum_{r=0}^{\infty}\frac{(-1)^r(ct) ^{{\mathcal{V}r}}}{\Gamma(\mathcal{V}r+1)}.
\end{align*}
     \end{theo}
\begin{proof}
    The bicomplex  Laplace  transform  of  the  Riemann-Liouville  integral  operator  of bicomplex order is given by
    \begin{align*}
        L\{D^{-\mathcal{V}}_tf(t);W\}=W^{-\mathcal{V}}F(W),
    \end{align*}
    where
    \begin{align*}
        F(W)=\int_0^\infty e^{-Wy}f(y)dy.
    \end{align*}
    Applying bicomplex Laplace transform to both side of \eqref{eq:n7}, we obtain
    \begin{align}\label{eq:j6}
        N(W)&=\frac{N_0}{W[1+(cW^{-1})^\mathcal{V}]}\nonumber\\
        &=\frac{N_0}{W}\sum_{r=0}^{\infty}(-1)^r(cW^{-1})^{\mathcal{V}r}\;\;(\mbox{provived}\; |W|_h>c)\nonumber\\
        &=N_0\sum_{r=0}^{\infty}\frac{(-1)^rc^{{\mathcal{V}r}}}{W^{\mathcal{V}r+1}}.
    \end{align}
    Now, taking the inverse Laplace transform of \eqref{eq:j6}, we get
    \begin{align*}
    N(t)=N_0\sum_{r=0}^{\infty}\frac{(-1)^r(ct) ^{{\mathcal{V}r}}}{\Gamma(\mathcal{V}r+1)}.
    \end{align*}
    Hence, the proof is completed.
\end{proof}
\begin{remark}
    Setting $\beta=0$ in Theorem \ref{th:a2}, we obtain the solution of fractional kinetic equation, which coincides with the results discussed in \cite{frac-kinetic}.
\end{remark}
\begin{theo}\label{th:h8}
    Assume that $\mathcal{V}=\alpha +j\beta=\nu_1e_1+\nu_2e_2,C=c_1e_1+c_2e_2 \in{\mathbb{BC}}$, $c>0$ with satisfies the condition $\Re(\alpha)>|\Im(\beta)|$, then the solution of the bicomplex fractional integral equation
    \begin{align}\label{eq:s8}
    N(t)-N_0\exp(Ct)=-c^\mathcal{V} D^{-\mathcal{V}}_tN(t),
     \end{align}
is given by
\begin{align*}
    N(t)=\sum_{r=0}^\infty N_0(-c^\mathcal{V})^kE_{\mathcal{V}k,C}(t).
\end{align*}
\end{theo}
\begin{proof}
    Applying the fractional integral operator $(-c^\mathcal{V})^rD^{-\mathcal{V}r}_t$ to both side of \eqref{eq:s8}, we obtain
    \begin{align*}
        (-c^\mathcal{V})^rD^{-\mathcal{V}r}_t N(t)-N_0(-c^\mathcal{V})^rD^{-\mathcal{V}r}_t\exp(Ct)=(-c^\mathcal{V})^{r+1} D^{-\mathcal{V}(r+1)}_tN(t),
    \end{align*}
    where $r\in\mathbb{N}\cup \{0\}$. Taking the sum from $r=0$ to $\infty$, we get
    \begin{align}\label{eq:d6}
        \sum_{r=0}^\infty(-c^\mathcal{V})^rD^{-\mathcal{V}r}_t N(t)- \sum_{r=0}^\infty N_0(-c^\mathcal{V})^rD^{-\mathcal{V}r}_t\exp(Ct)    &= \sum_{r=0}^\infty (-c^\mathcal{V})^{r+1} D^{-\mathcal{V}(r+1)}_tN(t)\nonumber\\
        &=\sum_{r=1}^\infty(-c^\mathcal{V})^r D^{-\mathcal{V}r}_tN(t).
    \end{align}
    After cancel out equal terms of \eqref{eq:d6}, we get
    \begin{align}\label{eq:f5}
        N(t)&=\sum_{r=0}^\infty N_0(-c^\mathcal{V})^rD^{-\mathcal{V}r}_t\exp(Ct)\nonumber\\
        &=\sum_{r=0}^\infty N_0(-c^\mathcal{V})^rD^{-\mathcal{V}r}_t\left[\sum_{n=0}^\infty\frac{(Ct)^n}{n!}\right]\nonumber\\
        &=\sum_{r=0}^\infty N_0(-c^\mathcal{V})^k\left[\sum_{n=0}^\infty\frac{C^n}{n!}D^{-\mathcal{V}k}_tt^n\right]
    \end{align}
    Using \eqref{eq:f5} and \eqref{eq:k5}, we have
    \begin{align*}
        N(t)&=\sum_{r=0}^\infty N_0(-c^\mathcal{V})^k\left[\sum_{n=0}^\infty\frac{C^n}{n!}\frac{\Gamma(n+1)}{\Gamma(n+\mathcal{V}k+1)}t^{n+\mathcal{V}k}\right]\\
        &=\sum_{r=0}^\infty N_0(-c^\mathcal{V})^kE_{\mathcal{V}k,C}(t).
    \end{align*}
    Hence, the proof is completed.
\end{proof}
\begin{theo}
     Suppose that $\mathcal{V}=\alpha +j\beta=\nu_1e_1+\nu_2e_2,\mu=\gamma+j\delta=\mu_1e_1+\mu_2e_2,C=c_1e_1+c_2e_2,Z_0=pe_1+qe_2 \in{\mathbb{BC}}$ with satisfies the conditions $\Re(\alpha)>|\Im(\beta)|$, $\Re(\gamma)+1>|\Im(\delta)|$ and $Z_0\notin\mathbb{O}_2$, then for the solution of the bicomplex fractional integral equation
    \begin{align}\label{eq:s6}
    N(t)-N_0E_{\mu k,C}(Z_0t)=-c^\mathcal{V} D^{-\mathcal{V}}_tN(t),
     \end{align}
the following result holds
\begin{align*}
    N(t)=\sum_{k=0}^\infty N_0(-c^\mathcal{V})^kZ_0^{-\mathcal{V}k}E_{(\mu+\mathcal{V}) k,C}(Z_0t).
\end{align*}
\end{theo}
\begin{proof}
Using idempotent representation of the Riemann-Lioville fractional integration of bicomplex oreder \cite{rlfo}, we obtain
\begin{align*}
    D^{-\mathcal{V}k}_tE_{\mu k,C}(Z_0t)&=D^{-\nu_1k}_tE_{\mu_1 k,c_1}(pt)e_1+D^{-\nu_2k}_tE_{\mu_2 k,c_2}(qt)e_2\\
    &=(p)^{\mu_1k}\sum_{n=0}^{\infty}\frac{(c_1p)^n}{\Gamma(\mu_1k+n+1)}D^{-\nu_1k}_tt^{\mu_1k+n}e_1\\&\quad+(q)^{\mu_2k}\sum_{n=0}^{\infty}\frac{(c_2q)^n}{\Gamma(\mu_2k+n+1)}D^{-\nu_2k}_tt^{\mu_2k+n}e_2\\
    &=(p)^{\mu_1k}\sum_{n=0}^{\infty}\frac{(c_1p)^n}{\Gamma(\mu_1k+\nu_1k+n+1)}t^{\mu_1k+\nu_1k+n}e_1\\&\quad+(q)^{\mu_2k}\sum_{n=0}^{\infty}\frac{(c_2q)^n}{\Gamma(\mu_2k+\nu_2k+n+1)}t^{\mu_2k+\nu_2k+n}e_2\\
    &=(p)^{-\nu_1k}E_{(\mu_1+\nu_1) k,c_1}(pt)e_1+(q)^{-\nu_2k}E_{(\mu_2+\nu_2) k,c_2}(qt)e_2\\
    &=Z_0^{-\mathcal{V}k}E_{(\mu+\mathcal{V}) k,C}(Z_0t).
\end{align*}
  Proceeding in a similar manner of Theorem\ref{th:h8}, it is observed that
  \begin{align*}
      N(t)&=\sum_{k=0}^\infty N_0(-c^\mathcal{V})^kD^{-\mathcal{V}k}_tE_{\mu k,C}(Z_0t)\nonumber\\
      &=\sum_{k=0}^\infty N_0(-c^\mathcal{V})^kZ_0^{-\mathcal{V}k}E_{(\mu+\mathcal{V}) k,C}(Z_0t).
  \end{align*}
  Hence, the proof is completed.
\end{proof}
 The Riemann-Lioville fractional differentiation of bicomplex order for function $f(t)=t^u$ is given by \cite{rlfo} :
 \begin{align}\label{eq:kp}
     D^M_tt^u=\frac{\Gamma(u+1)}{\Gamma(u-M+1)}t^{u-M},
 \end{align}
 where $M=m_1+jm_2\in\mathbb{BC}$ with $\Re(m_1)>0$ and $u>-1$.
 \begin{theo}\label{th:6a}
     Let $Z_0=pe_1+qe_2, \mathcal{V}=\alpha +j\beta=\nu_1e_1+\nu_2e_2,C=c_1e_1+c_2e_2 \in{\mathbb{BC}}$ with satisfies the conditions $\alpha\in\mathbb{R}$, $\beta\in i\mathbb{R}$ and $\alpha+1>|\beta|$. Then Riemann-Lioville fractional derivative of bicomplex Miller Ross function is given by
     \begin{align*}
          D^M_t E_{\mathcal{V},C}(Z_0t)=Z_0^ME_{\mathcal{V}-M,C}(Z_0t),
     \end{align*}
     where $M=m_1+jm_2\in\mathbb{BC}$ with $\Re(m_1)>0$.
 \end{theo}
 \begin{proof}
     Using \eqref{eq:gh} and \eqref{eq:kp}, we have
     \begin{align*}
         D^M_t E_{\mathcal{V},C}(Z_0t)&=Z_0^\mathcal{V}\sum_{r=0}^{\infty}\frac{(CZ_0)^r}{\Gamma(\mathcal{V}+r+1)}D^Mt^{\mathcal{V}+r}\\&
         =Z_0^\mathcal{V}\sum_{r=0}^{\infty}\frac{(CZ_0)^r}{\Gamma(\mathcal{V}+r+1)}[D^Mt^{\nu_1+r}e_1+D^Mt^{\nu_2+r}e_2]\\
         &=Z_0^\mathcal{V}\sum_{r=0}^{\infty}\frac{(CZ_0)^r}{\Gamma(\mathcal{V}+r+1)}\left[\frac{\Gamma(\nu_1+r+1)}{\Gamma(\nu_1+r-M+1)}t^{\nu_1+r-M}e_1+\frac{\Gamma(\nu_2+r+1)}{\Gamma(\nu_2+r-M+1)}t^{\nu_2+r-M}e_2\right]\\
         &=Z_0^\mathcal{V}\sum_{r=0}^{\infty}\frac{(CZ_0)^r}{\Gamma(\mathcal{V}+r+1)}\frac{\Gamma((\mathcal{V}+r+1)}{\Gamma(\mathcal{V}+r-M+1)}t^{\mathcal{V}+r-M}\\
         &=Z_0^\mathcal{V}\sum_{r=0}^{\infty}\frac{(CZ_0)^r}{\Gamma(\mathcal{V}+r-M+1)}t^{\mathcal{V}+r-M}\\
         &=Z_0^\mathcal{V}t^{\mathcal{V}-M}\sum_{r=0}^{\infty}\frac{(CZ_0t)^r}{\Gamma(\mathcal{V}+r-M+1)}\\
         &=Z_0^ME_{\mathcal{V}-M,C}(Z_0t).
     \end{align*}
     Hence, the proof is completed.
 \end{proof}
  \begin{theo}\label{th:5s}
      Suppose that $Z_0, \mathcal{V}=\alpha +j\beta=\nu_1e_1+\nu_2e_2,C=c_1e_1+c_2e_2 \in{\mathbb{BC}}$ with satisfies the conditions $\alpha\in\mathbb{R}$, $\beta\in i\mathbb{R}$ and $\alpha+1>|\beta|$. Then $y_q=\sum\limits_{k=1}^qC^{\frac{(k-1)p}{q}}E_{\mathcal{V}-p+\frac{kp}{q},C}(Z_0t)$ satisfies the following fractional differential equation
      \begin{align*}
          [D^{\frac{p}{q}}_t-(Z_0C)^{\frac{p}{q}}]y_q=\sum\limits_{r=0}^{p-1}\frac{C^rZ_0^{\mathcal{V}+r-p+\frac{p}{q}}t^{\mathcal{V}+r-p}}{\Gamma(\mathcal{V}+r-p+1)}.
      \end{align*}
 \end{theo}
 \begin{proof}
     Setting $M=\frac{p}{q}$ in Theorem \ref{th:6a}, we obtain
     \begin{align}\label{eq:l4}
         D^\frac{p}{q}_t E_{\mathcal{V},C}(Z_0t)=Z_0^\frac{p}{q}E_{\mathcal{V}-\frac{p}{q},C}(Z_0t).
     \end{align}
 Performing the replacement $\mathcal{V}$ by $\mathcal{V}-p+\frac{kp}{q}$ in \eqref{eq:l4}, we get
 \begin{align*}
      D^\frac{p}{q}_t E_{\mathcal{V}-p+\frac{kp}{q},C}(Z_0t)=Z_0^\frac{p}{q}E_{\mathcal{V}-p+\frac{(k-1)p}{q},C}(Z_0t).
 \end{align*}
 Now
 \begin{align*}
     D^\frac{p}{q}_ty_q&=\sum\limits_{k=1}^qC^{\frac{(k-1)p}{q}} D^\frac{p}{q}_tE_{\mathcal{V}-p+\frac{kp}{q},C}(Z_0t)\\
     &=Z_0^\frac{p}{q}\sum\limits_{k=1}^qC^{\frac{k-1}{n}}E_{\mathcal{V}-p+\frac{(k-1)p}{q},C}(Z_0t)\\
     &=Z_0^\frac{p}{q}\left[E_{\mathcal{V}-p,C}(Z_0t)+\sum\limits_{k=2}^qC^{\frac{(k-1)p}{q}}E_{\mathcal{V}-p+\frac{(k-1)p}{q},C}(Z_0t)\right]\\
     &=Z_0^\frac{p}{q}\left[(Z_0t)^{\mathcal{V}-p}\sum\limits_{r=0}^\infty\frac{(CZ_0t)^r}{\Gamma(\mathcal{V}-p+r+1)}+\sum\limits_{k=2}^qC^{\frac{(k-1)p}{q}}E_{\mathcal{V}-p+\frac{(k-1)p}{q},C}(Z_0t)\right]\\
     &=Z_0^\frac{p}{q}\left[\sum\limits_{r=0}^{p-1}\frac{C^r(Z_0t)^{\mathcal{V}+r-p}}{\Gamma(\mathcal{V}-p+r+1)}+\sum\limits_{r=p}^\infty\frac{C^r(Z_0t)^{\mathcal{V}+r-p}}{\Gamma(\mathcal{V}-p+r+1)}+\sum\limits_{k=2}^qC^{\frac{(k-1)p}{q}}E_{\mathcal{V}-p+\frac{(k-1)p}{q},C}(Z_0t)\right]\\
     &=Z_0^\frac{p}{q}\sum\limits_{r=0}^{p-1}\frac{C^r(Z_0t)^{\mathcal{V}+r-p}}{\Gamma(\mathcal{V}-p+r+1)}+Z_0^\frac{p}{q}\left[\sum\limits_{r=0}^\infty\frac{C^{r+p}(Z_0t)^{\mathcal{V}+r}}{\Gamma(\mathcal{V}+r+1)}+\sum\limits_{k=2}^qC^{\frac{(k-1)p}{q}}E_{\mathcal{V}-p+\frac{(k-1)p}{q},C}(Z_0t)\right]\\
     &=Z_0^\frac{p}{q}\sum\limits_{r=0}^{p-1}\frac{C^r(Z_0t)^{\mathcal{V}+r-p}}{\Gamma(\mathcal{V}-p+r+1)}+Z_0^\frac{p}{q}\left[C^pE_{\mathcal{V},C}(Z_0t)+\sum\limits_{k=2}^qC^{\frac{(k-1)p}{q}}E_{\mathcal{V}-p+\frac{(k-1)p}{q},C}(Z_0t)\right]\\
     &=\sum\limits_{r=0}^{p-1}\frac{C^rZ_0^{\mathcal{V}+r-p+\frac{p}{q}}t^{\mathcal{V}+r-p} }{\Gamma(\mathcal{V}-p+r+1)}+(CZ_0)^\frac{p}{q}\left[C^{\frac{(q-1)p}{q}}E_{\mathcal{V},C}(Z_0t)+\sum\limits_{k=2}^{q}C^{\frac{(k-2)p}{q}}E_{\mathcal{V}-p+\frac{(k-1)p}{q},C}(Z_0t)\right]\\
     &=\sum\limits_{r=0}^{p-1}\frac{C^rZ_0^{\mathcal{V}+r-p+\frac{p}{q}}t^{\mathcal{V}+r-p}}{\Gamma(\mathcal{V}-p+r+1)}+(CZ_0)^\frac{p}{q}\left[C^{\frac{(q-1)p}{q}}E_{\mathcal{V},C}(Z_0t)+\sum\limits_{k=1}^{q-1}C^{\frac{(k-1)p}{q}}E_{\mathcal{V}-p+\frac{kp}{q},C}(Z_0t)\right]\\
     &=\sum\limits_{r=0}^{p-1}\frac{C^rZ_0^{\mathcal{V}+r-p+\frac{p}{q}}t^{\mathcal{V}+r-p}}{\Gamma(\mathcal{V}-p+r+1)}+(CZ_0)^\frac{p}{q}y_q.
 \end{align*}
 Hence the proof is completed.
 \end{proof}
 \begin{corollary}
     Setting $p=1$ and $q=n$ in Theorem \ref{th:5s}, we obtain $y_n=\sum\limits_{k=1}^nC^{\frac{k-1}{n}}E_{\mathcal{V}-1+\frac{k}{n},C}(Z_0t)$ satisfies the following fractional differential equation
      \begin{align*}
          [D^{\frac{1}{n}}_t-(Z_0C)^{\frac{1}{n}}]y_n=\frac{Z_0^{\mathcal{V}-1+\frac{1}{n}}t^{\mathcal{V}-1}}{\Gamma(\mathcal{V})}.
      \end{align*}
 \end{corollary}
 \section{Conclusion}
 In this paper, we introduce bicomplex Miller Ross function and explore its properties, building upon its complex counterpart. We derive various properties and special cases of the bicomplex Miller Ross function including recurrence relations, integral representations, differential relations and differential equation.  Furthermore, we introduce a new bicomplex generalization of the fractional kinetic equation involving the bicomplex Miller Ross function and derive its solution. The bicomplex generalised fractional kinetic equation using the bicomplex Miller Ross function can be reduced to fractional kinetic equations involving other well-known functions under specific special conditions. Future research could establish the relationship of this bicomplex generalization with fractional calculus operators and investigate its applications in various fields of science and engineering.

\end{document}